\documentclass{amsart}
\usepackage{import}
\input{preamble.sty}

\title{Separability in Morse local-to-global groups}
\author{Lawk Mineh and Davide Spriano}

\address{School of Mathematics, University of Southampton, Southampton, SO17 1BJ, UK}
\email{L.Mineh@soton.ac.uk}
\address{Mathematical Institute, University of Oxford, Oxford, OX2 6GG, UK}
\email{spriano@maths.ox.ac.uk}

\begin{document}

\begin{abstract}
   We show that in a Morse local-to-global group where stable subgroups are separable, the product of any stable subgroups is separable. As an application, we show that the product of stable subgroups in virtually special groups is separable. 
\end{abstract}

\maketitle

\section{Introduction}
Given a group \(G\), we can equip it with the \emph{profinite topology}, whose basic open subsets are cosets of finite index subgroups of \(G\).
A subset of \(G\) is said to be \emph{separable} if it is closed in the profinite topology on \(G\).
The group \(G\) is called \emph{residually finite} if the trivial subgroup is separable in \(G\).

Knowing that particular subsets of groups are separable often gives useful information about the group.
For example, in a finitely presented group, separability of a finitely generated subgroup gives a solution to the membership problem for that subgroup.
In a geometric setting, separability properties of the fundamental group of a space correspond to desirable lifting properties of that space:
immersed subcomplexes of a complex \(X\) may be promoted to embedded ones in a finite sheeted cover of \(X\), provided that their corresponding subgroups are separable in \(\pi_1 X\).
For an example involving subsets rather than subgroups, it has been proven that if \(X\) is a nonpositively curved cube complex in which every double coset of hyperplane stabilisers is separable in \(\pi_1 X\), then \(X\) has a finite sheeted special cover \cite{Haglund-Wise}.

It is a difficult problem to show that a given subset of a group is separable, especially when one is only given some geometric data about the group.
For instance, even the question of whether hyperbolic groups are residually finite is a long-standing open problem.
It is known that all hyperbolic groups are residually finite if and only if every quasiconvex subgroup is separable in every hyperbolic group \cite{AgolGrovesManning}.
Minasyan showed that if \(G\) is a hyperbolic group in which every quasiconvex subgroup is separable, the setwise product of any finite number of quasiconvex subgroups is also separable in \(G\) \cite{MinasyanGFERF}, extending a result of Ribes and Zalesskii, who proved the same result in the case \(G\) is free \cite{RibesZalesskii}.
Recently, the first author and Minasyan provided generalisation of this result in the setting of relatively hyperbolic groups \cite{MinMin}.
In this paper, we will provide another natural generalisation of this product separability result to the class of groups with the \emph{Morse local-to-global} (MLTG) property.

Introduced in \cite{MLTG}, the MLTG property roughly speaking requires that quasigeodesics with hyperbolic-like properties behave similarly to quasigeodesics in hyperbolic spaces. 
Consider the following two perspectives on hyperbolic spaces. 
The first involves Morse geodesics: we say that a quasigeodesic is \emph{Morse} if any other quasigeodesic with the same endpoints stays uniformly finite Hausdorff-distance from it (see Definition \ref{def:MorseQG}). 
It is a well-known fact that every quasigeodesic in a hyperbolic space satisfies the Morse property, and moreover that a space is hyperbolic if and only if all of its geodesics are uniformly Morse \cite{BonkQuasigeodesic} (for a discussion on uniformity, see Section 10 in \cite{CashenMackay}).
This motivates the study of Morse quasigeodesics in spaces that are not hyperbolic, an approach that has been successful in understanding the properties of spaces up to quasi-isometry \cite{CordesHumeStability,IncertiMedici,GraeberKarrerLazarovich,CordesRussellSprianoZalloum,QuingRafi}. 
On the other hand, Gromov showed that a space is hyperbolic if and only if all of it local quasigeodesics, i.e. paths that are quasigeodesics on every subpath of a certain length (see Definition \ref{def:LocalProperty}) are globally quasigeodesics. 
The Morse local-to-global property puts these two perspectives together and prescribes that all paths that are locally Morse quasigeodesics are globally Morse quasigeodesics. 

In groups with the MLTG property, elements acting on Morse geodesics behave ``as they should''. 
For instance, it is appealing to think that given two independent infinite order elements with Morse axes, then it is possible to use a ping-pong argument to generate a free subgroup using these elements. 
In general finitely generated groups this is not true, and we require the MLTG property in order to run such arguments. 
The above suggests that the failure of the MLTG property seems to imply some pathological behaviour.
Indeed, the only known examples of groups without the MLTG property are not finitely presentable. 
On the other hand, many well-behaved classes of groups, such as 3-manifold groups, CAT(0) groups, and mapping class groups are known to satisfy the MLTG property.

Our main theorem is concerned with separability of products of \emph{stable subgroups}. 
Stable subgroups were introduced by Durham and Taylor, who showed that the convex cocompact subgroups of the mapping class groups are precisely the stable ones \cite{DurhamTaylorStability}. 
For infinite cyclic subgroups the notion of stability and fixing a Morse quasigeodesic agree, and in general stable subgroups present many properties akin to quasiconvex subgroups of hyperbolic groups. 

\begin{thm}
\label{thm:MTLG_stable_product_sep}
    Let \(G\) be a finitely generated group with the Morse local-to-global property, and suppose that any stable subgroup of \(G\) is separable.
    Then the product of any stable subgroups of \(G\) is separable.
\end{thm}

Recall that a group is \emph{LERF} (locally extended residually finite) if every finitely generated subgroup is separable.
The following statement may be of more general interest, for instance as a criterion for showing that a given group is not LERF.
As stable subgroups are always finitely generated (see Lemma~\ref{lem:props_of_stable_subgps}), the hypotheses are stronger than the above theorem.

\begin{cor}
    Let \(G\) be a finitely generated LERF group with the Morse local-to-global property. Then the product of any stable subgroups of \(G\) is separable.
\end{cor}

A group is \emph{virtually special} if it is has a finite index subgroup that is the fundamental group of a special cube complex.
Triple cosets of convex subgroups in virtually special groups are known to be separable \cite{Shepherd_Imitator_homs}.
We extend this result to arbitrary products of stable subgroups, which are quasiconvex.

\begin{cor}
\label{cor:special_stable_product}
    Let \(G\) be a virtually special group. Then the product of any stable subgroups of \(G\) is separable.
\end{cor}

\emph{Strongly quasiconvex} subgroups, also known as \emph{Morse} subgroups, were introduced independently by Genevois and Tran \cite{GenevoisHyperbolicities,TranOnStrongly}, and Tran showed that a subgroup of a finitely generated group is stable if and only if it is strongly quasiconvex and hyperbolic \cite[Proposition~4.3]{TranOnStrongly}.
In the case of right-angled Artin groups, which contain many stable subgroups \cite{Stable_subgroups_RAAGS}, we can use \cite[Corollary~7.4]{RSTConvexity} to deduce the following. 

\begin{cor}
\label{cor:raag_strongly_qc_product}
    Suppose that \(\Gamma\) is a finite connected graph, and let \(A_\Gamma\) be the associated right-angled Artin group.
    Then the product of any strongly quasiconvex subgroups of \(A_\Gamma\) is separable.
\end{cor}

\subsection{Acknowledgments}
The authors are grateful to the support of the CRM in Montr\'{e}al where part of this work was completed. We would like to thank Sam Shepherd for interesting conversations on the topic, and the anonymous referee for their careful reading of the paper.


\section{Preliminaries}

Let us establish some notational conventions.
Given a group \(G\) and subgroup \(H \leqslant G\), we will write \(H \leqslant_f G\) when \(H\) has finite index in \(G\).
If \(g \in G\), we will use \(H^g\) to denote the conjugate subgroup \(gHg^{-1}\).

For a metric space \(X\) and points \(x, y, z \in X\), we will write
\[
    \langle x, y \rangle_z = \frac{1}{2} \Big( \dist(x,z) + \dist(y,z) - \dist(x,y) \Big)
\]
for the \emph{Gromov product} of \(x\) and \(y\) with respect to \(z\).
When \(X\) is the Euclidean plane, \(\langle x, y \rangle_z\) is exactly the distance between \(z\) and the points of tangency on an incircle for the triangle with vertices \(x, y\), and \(z\).
The Gromov product thus acts as a vague analogue for the notion of the angle spanned by two geodesics issuing from a single point in a metric space.

In this paper, we will restrict our attention to Cayley graphs of groups. 
Let $G$ be a group and $S$ a generating set for $G$. The \emph{Cayley graph} of $G$ with respect to $S$ is the graph $\mathrm{Cay}(G,S)$ with vertex set $G$ and elements $g,h$ connected by an edge if either $gh^{-1}\in S$ or $hg^{-1}\in S$.

We equip the set of vertices of a graph with the metric induced by declaring all of its edges to have length one. 
For a Cayley graph \(\mathrm{Cay}(G,S)\), we write \(\dist_S\) for this metric, which is exactly the word metric on \(G\) with respect to \(S\).
For \(g \in G\), we will write \(\abs{g}_S = \dist_S(1,g)\).

Given a path \(\gamma\) of a graph, we will denote its length (i.e. number of edges) by \(\ell(\gamma)\).
Given metric spaces $(X, \dist_X)$ and $(Y,\dist_Y)$, a \emph{$(\lambda,c)$-quasi-isometric embedding} is a map $f\colon X \to Y$ such that the following holds for any pair $x,y\in X$.

\[\frac{1}{\lambda}\dist_Y(f(x), f(y)) - c \leq \dist_X(x,y) \leq \lambda \dist_Y(f(x), f(y)) + c.\]
A \emph{$(\lambda,c)$-quasigeodesic} is a $(\lambda,c)$-quasi-isometric embedding of an interval $I\subset \mathbb{R}$.

The main geometric definition of the paper is the Morse local-to-global property. To define it, we need to define the Morse property and what is means for a given property of a path to be local. 

\begin{defn}[Local property]\label{def:LocalProperty}
    A path $\gamma \colon I \to X$ is said to \emph{$L$-locally} satisfy a property $P$ if each subpath of the form $\gamma\vert_{[t_1, t_2]}$ with $t_2-t_1 \leq L$ satisfies $P$. 
    When a path $\gamma$ is $L$-locally a $(\lambda,c)$-quasigeodesic, we say that $\gamma$ is a $(L; \lambda,c)$-local quasigeodesic. 
\end{defn}

\begin{defn}[Morse quasigeodesic]\label{def:MorseQG}
    Let $M\colon \mathbb{R}_{\geq1}\times\mathbb{R}_{\geq0}\to \mathbb{R}_{\geq0}$ be a non-decreasing function.
    A quasigeodesic $\gamma\colon I \to X$ is \emph{$M$-Morse} if for any $(\lambda,c)$-quasigeodesic segment $\eta \colon [a,b] \to X$ such that $\eta(a) =\gamma(t_1), \eta(b) = \gamma(t_2)$ we have 
    \[\dist_{Haus}(\gamma\vert_{[t_1, t_2]}, \eta) \leq M(\lambda,c),\] 
    where $\dist_{Haus}$ denotes the Hausdorff distance. We say that $\gamma$ is an \emph{$(M; \lambda,c)$-Morse quasigeodesic}, and \(M\) is its \emph{Morse gauge}. 
\end{defn}

Morse geodesics in any geodesic space satisfy a thin triangles condition, similar to geodesics in a hyperbolic metric space.

\begin{lemma}[{\cite[Lemma 3.6]{Qing_Morse_boundary}}]
\label{lem:morse_thin_triangles}
    Let \(X\) be a geodesic metric space and suppose that \(p\) and \(q\) are \(M\)-Morse geodesics with \(p_- = q_-\).
    There is a constant \(\delta = \delta(M) \geq 0\) such that for any geodesic \(r\) with endpoints \(r_- = p_+\) and \(r_+ = q_+\), the geodesic triangle with sides \(p, q\), and  \(r\) is \(\delta\)-thin.
\end{lemma}

\begin{defn}[Local Morse quasigeodesic]
    We say that a path is an \emph{$(L; M; \lambda,c)$-local Morse quasigeodesic} if it is $L$-locally an $M$-Morse $(\lambda,c)$-quasigeodesic. 
\end{defn}

\begin{defn}[MLTG property]
    We say that a metric space $X$ satisfies the \emph{Morse local-to-global property}, for short MLTG property, if for any choice of Morse gauge $M$ and constants $\lambda \geq 1, c \geq 0$ there exist a scale $L>0$, a Morse gauge $M'$ and constants $\lambda' \geq 1, c' \geq 0$ such that every $(L;M; \lambda,c)$-local Morse quasigeodesic is a $(M';\lambda',c')$-Morse quasigeodesic.  
\end{defn}

The strength of the MLTG property is that it allows us to draw global conclusions from local conditions, as the next lemma shows.  

\begin{lemma}
\label{lem:concat_of_morse_geodesics}
    Let $p = p_1 \ast \dots \ast p_n$ be a concatenation of $M$-Morse geodesics in space $X$ with the MLTG property and let $a_i$ and $a_{i+1}$ be the ordered endpoints of $p_i$. For each $\varepsilon> 0$ there are constants $B \geq 0, \lambda \geq 1, c \geq 0$, and a gauge $N$ (all depending only on $M$ and \(\varepsilon\)) such that if we have $\ell(p_i) > B$ for all $i = 2, \dots , n-1$ and $\langle(a_{i-1}), (a_{i+1})\rangle_{a_i} \leq \varepsilon$ for all $i = 2, \dots , n$, then $p$ is a $(N;\lambda,c)$-Morse quasigeodesic.
\end{lemma}

\begin{proof}
    We will show that there exists $M'$ depending only on \(M\) and $\varepsilon$ such that $p$ is locally a $(M'; 1, 2\epsilon)$-Morse quasigeodesic, and then we will choose an appropriate $B$ to use the MLTG property. 
    Given the existence of such \(M'\), the MLTG property gives us a Morse gauge \(N\) and constants \(\lambda \geq 1, c \geq 0, L \geq 0\) such that any \((L;M';1,2\varepsilon)\)-local Morse quasigeodesic is also a \((N;\lambda,c)\)-Morse quasigeodesic.
    
    We start with the quasigeodesic claim. 
    Take \(B \geq L + \varepsilon\) and observe that if we consider two points $x,y$ at parameterised distance at most $L$, they either lie on the same segment $p_i$ (which is geodesic), or on two consecutive segments $p_{i-1}$ and  $p_{i}$. 
    In the latter case, we have 
    \[
        \langle x, y\rangle_{a_i}\leq \langle a_{i-1}, a_{i+1}\rangle_{a_i} \leq \varepsilon
    \] 
    which means 
    \[
        \dist(x,y) + 2\epsilon \geq \dist(a_i,x) + \dist(a_i, y)= \ell(p_{i-1}\vert_{[x,a_i]}\ast p_i\vert_{[a_i,y]}).
    \] 

    Thus, $p$ is an $(L; 1, 2\epsilon)$-local quasigeodesic. 
    A similar computation to that above (with \(x = a_{i-1}, y = a_{i+1}\)) shows that 
    \[
        \dist(a_{i-1},a_{i+1}) \geq 2B - 2 \varepsilon > L
    \]
    where the last inequality comes from the choice of \(B\).
    Now, applying \cite[Lemma 2.15]{MLTG} to each concatenation $p_i\ast p_{i+1}$, we obtain some $M'$ (depending only on \(M\) and \(\varepsilon\)) such that $p$ is an $(L; M'; 1, 2\epsilon)$-local Morse quasigeodesic. 
    Now applying the MLTG property shows that \(p\) is $(N;\lambda,c)$-Morse quasigeodesic.
\end{proof}

The property of stability generalises the notion of having the Morse property from quasigeodesics to arbitrary subgroups.

\begin{defn}[Stable subgroup]
    Let \(G\) be a group with finite generating set \(S\), and let \(M\) be a Morse gauge and \(\mu \geq 0\) a constant.
    A subgroup \(H \leqslant G\) is called \emph{\((M,\mu)\)-stable} if any geodesic in \(\mathrm{Cay}(G,S)\) with endpoints in \(H\) is $M$-Morse and lies in the \(\mu\)-neighbourhood of \(H\).
    A subgroup is called \emph{stable} if it is \((M,\mu)\)-stable for some Morse gauge \(M\) and \(\mu \geq 0\).
\end{defn}

An immediate consequence of the definition is that a stable  subgroup of a finitely generated group is undistorted. We note that while the gauge \(M\) and constant \(\mu\) in the above depend on the choice of generating set \(S\), the property of being stable does not (see, for example, \cite[Lemma 3.4]{DurhamTaylorStability}). We recall some basic properties of stability, which follow from \cite{DurhamTaylorStability}.

\begin{lemma}
\label{lem:props_of_stable_subgps}
    Let \(G\) be a group with finite generating set \(S\) and suppose \(H \leqslant G\) is \((M,\mu)\)-stable.
    Then the following are true:
    \begin{enumerate}
        \item \label{item: passage to subgroups} if \(K \leqslant_f H\), then \(K\) is \((M,\mu')\)-stable for some \(\mu' \geq 0\);
        \item if \(g \in G\), then \(gHg^{-1}\) is stable;
        \item \(H\) is finitely generated and undistorted in \(G\);
        \item \(H\) is hyperbolic.
    \end{enumerate}
   \end{lemma}

The following lemma tells us that stable subgroups cannot be ``too parallel'' away from their intersection. 
More precisely, that there is a uniform upper bound on the Gromov product of elements from one subgroup when taken with minimal length coset representatives of the other. 

\begin{lemma}
\label{lem:stable_gromov_product_bound}
    Let \(G\) be a group with finite generating set \(S\), and suppose that \(H\) and \(K\) are \((M,\mu)\)-stable subgroups of \(G\).
    There is a constant \({\rho} = {\rho}(M,\mu,S) \geq 0\) such that if \(h \in H\) is a shortest (with respect to \(S\)) representative of its right coset \((H \cap K)h\), then for any \(k \in K\), we have \(\langle h, k \rangle_1 \leq {\rho}\). 
\end{lemma}

\begin{proof}
    Suppose for a contradiction that we can find elements $h\in H, k\in K$ such that $h$ is a shortest coset representative of $h(H\cap K)$ and $\langle h,k\rangle_1$ is arbitrarily large. 
    Since $H$ and $K$ are stable, any choice of geodesics $p = [1,h]$ and $q = [1,k]$ are \(M\)-Morse and lie in a \(\mu\)-neighbourhood of \(H\) and \(K\) respectively. 

    Let \(a_1, \dots, a_n\) be the vertices of \(p\) with \(\dist_S(1,a_i) \leq \langle h, k \rangle_1\).
    The assumption that \(\langle h, k \rangle_1\) can be taken to be arbitrarily large means that \(n\) can be taken to be arbitrarily large.
    Corresponding to each vertex \(a_i\), there is \(v_i \in H\) such that \(\dist_S(a_i, v_i) \leq \mu\) by stability of \(H\).
    Moreover, by Lemma~\ref{lem:morse_thin_triangles} there is \(\delta = \delta(M) \geq 0\) such that \(\dist_S(a_i, q) \leq \delta\) for each \(i = 1, \dots, n\).
    By stability of \(K\) and the triangle inequality, therefore, we obtain that \(\dist_S(v_i,K) \leq 2\mu + \delta\) for each \(i = 1, \dots, n\).
    Note that since \(p\) is geodesic
    \begin{equation}
    \label{eq:bd_on_vi_vj}
        \dist_S(1,v_i) \leq i - \mu \qquad \text{ and } \qquad \dist_S(v_i, h) \leq \dist_S(1,h) - i - \mu
    \end{equation}
    for each \(i = 1, \dots, n\).
    
    For each $i= 1, \dots, n$, let $g_i$ be the shortest element of \(G\) with respect to \(S\) such that $v_ig_i\in K$, so \(\abs{g_i}_S \leq 2\mu + \delta\). 
    Let \(N\) be the number of elements in the ball of radius \(2\mu + \delta\) about the identity in \(\mathrm{Cay}(G,S)\).
    Taking \(n\)  to be sufficiently large with respect to \(N\) and \(\mu\), there must be some pair \((i, j)\) with \(g_i = g_j\) satisfying \(j-i > 2\mu\).
    Then equation (\ref{eq:bd_on_vi_vj}) gives
    \begin{equation}
    \label{eq:vi+vj<h}
        \dist_S(1, v_i) + \dist_S(v_j, h) < \dist_S(1,h).
    \end{equation} 
    But then $v_jg_j(v_ig_i)^{-1} = v_jv_i^{-1} \in H\cap K$, as $v_i, v_j\in H$ and $v_ig_i, v_jg_j \in K$.
    Moreover 
    \begin{align*}
        \dist_S(v_jv_i^{-1}, h) &\leq \dist_S(v_jv_i^{-1},v_j) + \dist_S(v_j,h) \\
        &= \dist_S(1,v_i) + \dist_S(v_j,h) < \dist_S(1,h)
    \end{align*}
    where the last inequality is an application of (\ref{eq:vi+vj<h}). It follows that $v_iv_j^{-1} h\in (H\cap K) h$ and 
    \[\abs{v_i v_j^{-1} h}_S = \dist_S(v_jv_i^{-1}, h) < \dist_S(h,1) = \abs{h}_S\]

    which contradicts the fact that \(h\) is a minimal length representative of its \((H \cap K)\)-coset.
    Thus, there must be an upper bound on the Gromov product.
\end{proof}

We finish this section by recalling the key property that we need for our proof, namely that the MLTG property allows a ping-pong type argument for stable subgroups.

\begin{prop}[{\cite[Theorem 3.1]{MLTG}}]
\label{prop:stable_combination_thm}
    Let \(G\) be a group with finite generating set \(S\) and suppose \(G\) has the Morse local-to-global property.
    Let \(Q, R \leqslant G\) be \((M,\mu)\)-stable subgroups of \(G\).
    There is a constant \(C = C(M,\mu,S) \geq 0\) such that the following is true.

    Let \(Q' \leqslant Q\) and \(R' \leqslant R\) be subgroups such that \(Q' \cap R' = Q \cap R\) and \(\abs{g}_S \geq C\) for each \(g \in (Q' \cup R') \setminus (Q' \cap R')\).
    Then \(\langle Q', R' \rangle \cong Q' \ast_{Q' \cap R'} R'\).
    Moreover, if \(Q'\) and \(R'\) are finitely generated and undistorted in \(G\), then \(\langle Q', R' \rangle\) is stable.
\end{prop}


\section{Separability of products}
In this section we will prove the main theorem. We start with some elementary observations about separable subsets.
\begin{rmk}
    If \(U \subseteq G\) is a separable subset of \(G\), then \(U^{-1}, gU\), and \(Ug\) are separable for any \(g \in G\).
\end{rmk}

\begin{rmk}
\label{rmk:sep_of_products}
    Let \(H_1, \dots, H_n \leqslant G\) be subgroups of \(G\) and let \(a_0, \dots, a_n \in G\).
    Observe that
    \[
        a_0 H_1 a_1 \dots a_{n-1} H_n a_n = H_1^{a_0} H_2^{a_0 a_1} \dots H_n^{a_0 \dots a_{n-1}} a_0 \dots a_n,
    \]
    which is a translate of a product of conjugates of the subgroups \(H_1, \dots, H_n\).
    The set \(a_0 H_1 a_1 \dots a_{n-1} H_n a_n\) is thus separable if and only if the product of subgroups \(H_1^{a_0} \dots H_n^{a_0 \dots a_{n-1}}\) is separable.

    In particular, suppose there is \(n \in \NN\) such that any product of \(n\) stable subgroups are separable in \(G\), and suppose \(H_1, \dots, H_n\) are stable subgroups of \(G\). 
    Lemma~\ref{lem:props_of_stable_subgps}(2) gives that \(H_i^{a_0 \dots a_{i-1}}\) is stable for each \(1 \leq i \leq n\), so that the set \(H_1^{a_0} \dots H_n^{a_0 \dots a_{n-1}}\) is a product of \(n\) stable subgroups.
    By the observation above we may conclude that the set \(a_0 H_1 a_1 \dots a_{n-1} H_n a_n\) is separable in this situation.
\end{rmk}

In order to exploit the geometric properties afforded by the MLTG property, it is useful to choose coset representative that are geometrically meaningful, which we can do by the following remark.

\begin{rmk} 
\label{rmk:shortest_reps_in_product_RHS}
    Suppose that \(S\) is a generating set for group \(G\) and let \(H_1, \dots, H_n \leqslant G\) be subgroups of \(G\).
    Given elements \(x_1 \in H_1, \dots, x_n \in H_n\), there are elements \(y_1 \in H_1, \dots, y_n \in H_n\) such that \(x_1 \dots x_n = y_1 \dots y_n\) and \(\abs{y_i}_S\) is minimal among elements of the coset \((H_{i-1} \cap H_{i})y_i\) for each \(1 < i \leq n\).

    Indeed, there is \(y_n \in H_n\) and \(z_n \in H_n \cap H_{n-1}\) such that \(x_n = z_ny_n\) and \(\abs{y_n}_S\) is minimal among elements of \((H_{n-1} \cap H_n)x_n = (H_{n-1} \cap H_n)y_n\).
    Similarly there is \(y_{n-1} \in H_{n_1}\) and \(z_{n-1} \in H_{n-1} \cap H_{n-2}\) such that \(x_{n-1}z_n = z_{n-1}y_{n-1}\) and \(y_{n-1}\) is a shortest representative of \((H_{n-2} \cap H_{n-1})x_{n-1}z_{n} = (H_{n-2} \cap H_{n-1})y_{n-1}\).
    We can proceed by finite induction to find elements \(y_2\in H_2, \dots, y_{n} \in H_{n}\) and \(z_2 \in H_2 \cap H_3, \dots, z_{n} \in H_{n-1} \cap H_n\) with the properties described above.
    Setting \(y_1 = x_1 z_{2}  \in H_1\) completes the observation.
\end{rmk}

We conclude the section by proving the main theorem of the paper and the related corollaries.

\begin{proof}[Proof of Theorem~\ref{thm:MTLG_stable_product_sep}]
    We proceed by induction on the number \(n\) of stable subgroups.
    The case \(n = 1\) is exactly the hypothesis that stable subgroups of \(G\) are separable, so let \(H_1, \dots, H_n\) be stable subgroups of \(G\) with \(n > 1\) and suppose that the product of any \(n-1\) stable subgroups of \(G\) is separable.
    
    Fix a finite generating set \(S\) of \(G\).
    By taking maxima of gauges and constants, we may assume without loss of generality that \(H_1, \dots, H_n\) are \((M,\mu)\)-stable.
    Let \(\rho = \rho(M,\mu,S)\) be the constant obtained from Lemma~\ref{lem:stable_gromov_product_bound}, and let \(C = C(M,\mu,S)\) be the constant of Proposition~\ref{prop:stable_combination_thm}.
    Let \(B, \lambda, c \geq 0\) be the constants and \(N\) the Morse gauge obtained from applying Lemma~\ref{lem:concat_of_morse_geodesics} with gauge \(M\) and constant \(\rho\).
    
    Suppose for a contradiction that the product \(H_1 \dots H_n\) is not separable, so that there is some \(g \notin H_1 \dots H_n\) belonging to the profinite closure of \(H_1 \dots H_n\).
    This means that \(g\) is contained in every separable subset containing \(H_1 \dots H_n\).
    For ease of reading, we will write \(Q = H_1, R = H_2,\) and \(T_i = H_{i+2}\) whenever \(1 \leq i \leq s = n-2\).
    By hypothesis \(Q\) and \(R\) are separable, and thus their intersection \(I = Q \cap R\) is also.
    Let \(\{N_i\}_{i \in \NN}\) be an enumeration of the finite index subgroups of \(G\) containing \(I\), and note that \(I = \bigcap_{i \in \NN}N_i\) as \(I\) is separable.
    For each \(i\), we write
    \[
        N'_i = \bigcap_{j=1}^{i} N_j
    \]
    so that \(\{N'_i\}_{i\in\NN}\) is a sequence of nested finite index subgroups of \(G\) containing \(I\) whose intersection is equal to \(I\).

    For each \(i \in \NN\), we define the set
    \[
        K_i = Q \langle Q'_i, R'_i \rangle R T_1 \dots T_{s}
    \]
    where \(Q'_i = N'_i \cap Q \leqslant_f Q\) and \(R'_i = N'_i \cap R \leqslant_f R\).
    Note that \(I \subseteq N'_i\) for each \(i \in \NN\), so that \(Q'_i \cap R'_i = I\).
    It is immediate from the definition that \(K_i \supseteq QRT_1 \dots T_s\) for each \(i \in \NN\).
    Our aim is to show that for sufficiently large \(i\), the set \(K_i\) is separable and excludes the element \(g\).

    Let us first show that \(K_i\) is separable when \(i\) is large.
    Indeed, for a given \(i\), let \(x_1, \dots, x_a\) be left coset representatives for \(Q'_i\) in \(Q\) and \(y_1, \dots, y_b\) be right coset representatives for \(R'_i\) in \(R\).
    Then we have
    \[
        K_i = \bigcup^a_{j=1} \bigcup^b_{k=1} x_j \langle Q'_i, R'_i \rangle y_k T_1 \dots T_s.
    \]
    Since \(I\) is the intersection of all \(N'_i\), there is an index \(i_0 \in \NN\) such that for any \(i \geq i_0\), any element \(n \in N'_i\) with \(\abs{n}_S \leq C\) belongs to \(I\).
    By (1) and (3) of Lemma~\ref{lem:props_of_stable_subgps}, \(Q'_i\) and \(R'_i\) are finitely generated and undistorted, so we may apply Proposition~\ref{prop:stable_combination_thm} to obtain that \(\langle Q'_i, R'_i \rangle\) is stable.
    Thus, by Remark~\ref{rmk:sep_of_products} and the induction hypothesis, \(K_i\) can be written as a finite union of separable subsets, and so \(K_i\) is separable whenever \(i \geq i_0\).

    We now show that there is \(i \in \NN\) such that \(g \notin K_i\).
    As \(g\) belongs to the profinite closure of \(QRT_1 \dots T_s\) and for each \(i \geq i_0\) the set \(K_i\) is a profinitely closed subset of \(G\) containing \(QRT_1 \dots T_s\), \(g\) belongs to \(K_i\) for each \(i \geq i_0\).
    That is, for each \(i \geq i_0\) we may write
    \begin{equation}
    \label{eq:product_rep_of_g}
        g = q^{(i)} x_1^{(i)} \dots x_{m_i}^{(i)} r^{(i)} t_1^{(i)} \dots t_s^{(i)}
    \end{equation}
    for some  \(m_i \in \NN\) and \(x_j^{(i)} \in Q'_i \cup R'_i\) for each \(1 \leq j \leq m_i\), and where \( q^{(i)} \in Q, r^{(i)} \in R, t_1^{(i)} \in T_1, \dots, t_s^{(i)} \in T_s\).

    The remainder of the argument may be split into two essentially different cases based on the lengths of the elements obtained above: we summarise them here.
    
    In one case, the lengths of infinitely many of the elements \(r^{(i)}, t_1^{(i)}, \dots, t_{s-1}^{(i)}\) remain bounded as \(i\) tends to infinity.
    When this happens, we may pass to a subsequence where one these terms is constant in \(i\).
    This reduces the number of subgroups we have to consider in the product and we may apply the induction hypothesis to obtain our contradiction.
    The other situation to consider is when the lengths of these elements increase without bound.
    In this case, for sufficiently large values of \(i\) the products as in (\ref{eq:product_rep_of_g}) define paths that are (arbitrarily long) local Morse quasigeodesics.
    The MLTG property then shows that these are actually Morse quasigeodesics, resulting in another contradiction.

    \medskip
    
    \underline{\emph{Case 1:}}
        \(\liminf_{i \to \infty} \abs{r^{(i)}}_S < \infty\) or \(\liminf_{i \to \infty} \abs{t_j^{(i)}}_S < \infty\) for some \(1 \leq j < s\).

        We consider only the possibility that \(\liminf_{i \to \infty} \abs{r^{(i)}}_S < \infty\), for the other cases can be dealt with identically.
        It follows from this assumption that there is a subsequence of \((r^{(i)})_{i \in \NN}\) whose terms have length bounded by some fixed constant.
        Since there are only finitely many elements of \(G\) with any given length with respect to \(S\), we may pass to a further subsequence whose terms are all equal to a single element \(r \in R\).
        Hence we have
        \begin{equation}
        \label{eq:g_in_reduced_product}
            g \in Q \langle Q'_i, R'_i \rangle r T_1 \dots T_s \;\; \textrm{for infinitely many }i \in \NN.
        \end{equation}
        
        Now by the induction hypothesis and Remark~\ref{rmk:sep_of_products}, the set \(QrT_1 \dots T_s\) is separable in \(G\).
        Since \(g \notin QRT_1 \dots T_s\), we have \(g \notin Q r T_1 \dots T_s\), and so there is \(N \lhd_f G\) such that \(g \notin NQrT_1 \dots T_s = QNrT_1 \dots T_s\).
        The subgroup \(IN \leqslant_f G\) is a finite index subgroup of \(G\) containing \(I\), so \(N'_{i_1} \subseteq IN\) for some \(i_1 \in \NN\).
        Since the sequence of subgroups \(\{N'_i\}_{i \in \NN}\) is nested, we have thus shown that
        \[
            Q \langle Q'_i, R'_i \rangle r T_1 \dots T_s \subseteq Q N'_i r T_1 \dots T_s \subseteq Q I N r T_1 \dots T_s = Q N r T_1 \dots T_s
        \]
        for any \(i \geq i_1\), where the last equality uses the fact that \(QI = Q\).
        However, the fact that \(g \notin Q N r T_1 \dots T_s\) now contradicts the inclusions of (\ref{eq:g_in_reduced_product}), so this case is impossible.

    \medskip
    
    \underline{\emph{Case 2:}}
        \(\liminf_{i \to \infty} \abs{r^{(i)}}_S = \infty\) and \(\liminf_{i \to \infty} \abs{t_j^{(i)}}_S = \infty\) for all \(1 \leq j < s\).

        Define \(z_0 = 1, z_1 = q^{(i)}, z_2 = z_1 x_1^{(i)}, \dots, z_{m_i + 1} = z_{m_i}x_{m_i}^{(i)}, z_{m_i + 2} = z_{m_i + 1}r^{(i)}\), and \(z_{m_i + 3} = z_{m_i + 2}t_1^{(i)}, \dots, z_{m_i + 2 + s} = z_{m_i + s} t_s^{(i)}\).  For each \(0 \leq j \leq m_i + 1 + s\), we let \(p_j\) be a geodesic with \((p_j)_- = z_j\) and \((p_j)_+ = z_{j+1}\). Let \(p\) be the concatenation \(p_0 \ast \dots \ast p_{m_i + 1 + s}\) of these paths.

        We will use Lemma~\ref{lem:concat_of_morse_geodesics} to conclude that the path \(p\) is a uniform quasigeodesic. 
        Assuming this, the fact that \(\liminf_{i \to \infty} \abs{r^{(i)}}_S = \infty\) and \(\liminf_{i \to \infty} \abs{t_j^{(i)}}_S = \infty\) means that for sufficiently large \(i\), the distance between the endpoints of $p$ is greater than \(\abs{g}_S\), contradicting the fact that $p$ represents $g$. 

        Without loss of generality, we may assume \(x_1^{(i)} \in R'_i \setminus Q\) and \(x_{m_i}^{(i)} \in Q'_i \setminus R\), for otherwise we may replace \(q^{(i)}\) with \(q_1^{(i)} = q^{(i)} x_1^{(i)} \in Q\) and eliminate \(x_1^{(i)}\) from the product (and likewise with \(r^{(i)}\) and \(x_{m_i}^{(i)}\)).
        Further, we may assume by Remark~\ref{rmk:shortest_reps_in_product_RHS} that \(x_1^{(i)}, \dots, x_{m_i}^{(i)},\) and \(r^{(i)}\) are shortest representatives of their right \(I\)-cosets, and in particular \(x_1^{(i)}, \dots, x_{m_i}^{(i)}, r^{(i)} \notin I\). 
        Similarly we take \(t_1^{(i)}\) to be a shortest representative of \((R \cap T_1) t_1^{(i)}\) and, for \(1 < i \leq s\), the element \(t_j^{(i)}\) to be a shortest representative of $(T_{j-1} \cap T_{j})t_j^{(i)}$.

        The above paragraph together with Lemma~\ref{lem:stable_gromov_product_bound} shows that
        \begin{equation}
        \label{eq:zi_gromov_product_boudn}
            \langle z_{j-1}, z_{j+1} \rangle_{z_j} \leq {\rho} \quad \textrm{ for } j = 1, \dots, m_i + s.
        \end{equation}

        We now verify the hypotheses of Lemma~\ref{lem:concat_of_morse_geodesics}.
        Each of the geodesic segments \(p_i\) represents an element of a finite index subgroup of one of \(Q, R, T_1, \dots, T_{s-1},\) or \(T_s\),
        Therefore by Lemma~\ref{lem:props_of_stable_subgps}(1), we obtain that the geodesic segments $p_i$ are $M$-Morse. 
        For any given \(B' \geq B\) (recall that \(B\) is the constant of Lemma~\ref{lem:concat_of_morse_geodesics} applied with \(M\) and \(\rho\)) we deduce the following.
        Since \(\liminf_{i \to \infty} \abs{r^{(i)}}_S = \infty\) and \(\liminf_{i \to \infty} \abs{t_j^{(i)}}_S = \infty\) we have that \(\abs{r^{(i)}}_S > B'\) and \(\abs{t_j^{(i)}}_S > B'\) for each \(j = 1, \dots, s\) and sufficiently large \(i\).
        Moreover, since  $x_{j}^{(i)}\in  \left(Q'_i \cup R'_i\right)\setminus I\subseteq N_i'\setminus I$ and since $\bigcap N_i' = I$, for $i$ large enough we have  $ \abs{{ x_{j}^{(i)}}}_S> B'$. 
        Thus Lemma~\ref{lem:concat_of_morse_geodesics} implies that \(p\) is \((N; \lambda,c)\)-Morse quasigeodesic.
        Finally, choosing \(B'\) sufficiently large with respect to \(\lambda, c,\) and \(\abs{g}_S,\) gives us that the endpoints of \(p\) are a greater distance than \(\abs{g}_S\) apart, the desired contradiction.

        From the above, there is some \(i \in \NN\) such that \(K_i\) is separable, contains the product \(QRT_1 \dots T_s\), and excludes \(g\).
        Therefore the product \(QRT_1 \dots T_s = H_1 \dots H_n\) is separable.
\end{proof}

\begin{proof}[Proof of Corollary~\ref{cor:special_stable_product}]
    Stable subgroups are quasiconvex, and quasiconvex subgroups of virtually special groups are separable by \cite[Corollary 7.9]{Haglund-Wise}.
    Moreover, \(\operatorname{CAT}(0)\) groups have the MLTG property by \cite[Theorem D]{MLTG}.
    Therefore Theorem~\ref{thm:MTLG_stable_product_sep} applies to give the result.
\end{proof}

\begin{proof}[Proof of Corollary~\ref{cor:raag_strongly_qc_product}]
    Let \(H_1, \dots, H_n\) be strongly quasiconvex subgroups of \(A_\Gamma\).
    If there is some \(1 \leq i \leq n\) such that \(H_i\) has finite index, then the product \(H_1 \dots H_n\) is a union of finitely many cosets of \(H_i\).
    Since \(H_i\) has finite index, it is separable in \(A_\Gamma\), whence \(H_1 \dots H_n\) is separable.

    Now suppose that each of the subgroups \(H_1, \dots, H_n\) has infinite index in \(A_\Gamma\).
    By \cite[Corollary 7.4]{RSTConvexity}, they must be stable subgroups.
    Noting that \(A_\Gamma\) is \(\operatorname{CAT}(0)\) and hence has the MLTG property \cite[Theorem D]{MLTG}, the conclusion follows by applying Theorem~\ref{thm:MTLG_stable_product_sep}.
\end{proof}

\bibliographystyle{alpha}
\bibliography{bibtex}

\end{document}